\documentclass{amsart}
\usepackage{amssymb,amsmath,latexsym}
\usepackage{amsthm}
\usepackage{fontenc}
\usepackage{amssymb}
\DeclareMathOperator{\erf}{erf}
\DeclareMathOperator{\erfc}{erfc}
\DeclareMathOperator{\erfi}{erfi}
\DeclareMathOperator{\Ei}{Ei}

\numberwithin{equation}{section}

\newtheorem{theorem}{Theorem}[section]

\begin{document}
\author{Alexander E Patkowski}
\title{On special Riemann xi function formulae of Hardy involving the digamma function}

\maketitle
\begin{abstract}We consider some properties of integrals considered by Hardy and Koshliakov, that have have connections to the digamma function. We establish a new general integral formula that provides a connection to the polygamma function. We also obtain lower and upper bounds for Hardy's integral through properties of the digamma function. \end{abstract}



\section{Introduction and Main formulas}
In a paper written by the well-known G. H. Hardy [10], an interesting integral formula is presented (corrected in [3])
\begin{equation} \int_{0}^{\infty}\frac{\Xi(t/2)}{1+t^2}\frac{\cos(xt)}{\cosh(\pi t/2)}dt=\frac{1}{4}e^{-x}\left(2x+\frac{1}{2}\gamma+\frac{1}{2}\log(\pi)+\log(2)\right)+\frac{1}{2}e^{x}\int_{0}^{\infty}\psi(t+1)e^{-\pi t^2e^{4x}}dt,\end{equation}
where $\psi(x):=\frac{\partial}{\partial x}\log(\Gamma(x)),$ $\Gamma(x)$ being the gamma function [1, pg.1], and $\Xi(t):=\xi(\frac{1}{2}+it),$ where [11, 15]
$$\xi(s):=\frac{1}{2}s(s-1)\pi^{-\frac{s}{2}}\Gamma(\frac{s}{2})\zeta(s).$$ Here we have used the standard notation for the Riemann zeta function $\zeta(s):=\sum_{n\ge1}n^{-s},$ for $\Re(s)>1.$ Koshliakov [12, eq.(14), eq.(20)] (or [3, eq.(1.15)]) produced this formula as well, but in a slightly different form,
\begin{equation} 2\int_{0}^{\infty}\frac{\Xi(t/2)}{1+t^2}\frac{\cos(xt)}{\cosh(\pi t/2)}dt=e^{x}\int_{0}^{\infty}(\psi(t+1)-\log(t))e^{-\pi t^2e^{4x}}dt.\end{equation}

Here the trick, which has been exploited in the many studies [3, 4, 5, 6, 7, 12, 14], is to re-write the left side of (1.1) as an inverse Mellin transform by utilizing the classical functional equation $\xi(s)=\xi(1-s)$ [15]. Titchmarsh has used simpler but similar integrals than the left side of (1.1) to obtain Hardy's result that $\Xi(t)$ has infinitely many real zeros [15]. See Csordas [2] for some more recent work in this direction. A recent paper by Dixit offers a beautiful generalization of formula (1.1), which involves a confluent hypergeometric function [3, Theorem 1.3]. Other refinements can be found in [4, 5]. \par The purpose of this paper is to offer some further results concerning (1.1) that appear to have been overlooked. A new generalization is offered as well that takes a different route than the ones offered by [3, 4, 5], considering what is essentially the polygamma function [9, pg.904, eq. NH 37(1)]. In the following section, we offer some new inequalities that we noticed upon examining the integral given in (1.1). \par Recall the Mellin transform of a suitable function $f(t)$ is given by [11, pg.90, eq.(4.105)]
$$\int_{0}^{\infty}f(t)t^{s-1}dt=F(s),$$ and its inverse,
$$f(t)=\frac{1}{2\pi i}\int_{c-i\infty}^{c+i\infty}F(s)t^{-s}ds,$$
provided $c$ is a real number chosen where $F(s)$ is analytic. We will be applying Parseval's formula for Mellin transforms throughout the proofs of our theorems [15, pg.34, eq.(2.15.10)]
$$\int_{0}^{\infty}f(t)g(t)dt=\frac{1}{2\pi i}\int_{c-i\infty}^{c+i\infty}F(s)G(1-s)ds.$$ Here the needed condition for validity is to have $c$ chosen as a real number restricted to the region where $F(s)G(1-s)$ is analytic.

\begin{theorem} For $m\in\mathbb{N}_{0},$ we have that 
$$(-1)^{m}e^{x}\int_{0}^{\infty}t^m \bar{\psi}_{m}(t)e^{-\pi t^2e^{4x}}dt$$
$$=\frac{1}{2}\int_{0}^{\infty}\frac{\Xi(t)}{\cosh(\pi t)(t^2+\frac{1}{4})}\left(\frac{\Gamma(\frac{1}{2}-it+m)}{\Gamma(\frac{1}{2}-it)}e^{-2xit}+\frac{\Gamma(\frac{1}{2}+it+m)}{\Gamma(\frac{1}{2}+it)}e^{2ixt}\right)dt,$$
$$\bar{\psi}_{m}(t):=\frac{\partial^m }{\partial t^m}(\psi(t+1)-\log(t)).$$

\end{theorem}
\begin{proof} 
This is a special application of Parseval's theorem for Mellin transforms with one function chosen as $f(t)=e^{-\pi t^2},$ and the other as $g(t)=t^m \bar{\psi}_{m}(t).$ From Titchmarsh [15, eq.(2.15.7)] we have for $0<c<1,$ 
\begin{equation}\psi(x+1)-\log(x)=-\frac{1}{2\pi i}\int_{c-i\infty}^{c+i\infty}\frac{\pi\zeta(1-s)}{\sin(\pi s)}x^{-s}ds.\end{equation}
Differentiating (1.3) $m$ times, gives us
\begin{equation}\bar{\psi}_{m}(x)=(-1)^{m-1}\frac{1}{2\pi i}\int_{c-i\infty}^{c+i\infty}\frac{\pi\zeta(1-s)}{\sin(\pi s)}\frac{\Gamma(s+m)}{\Gamma(s)}x^{-s-m}ds.\end{equation}
This may also be obtained from the integral formula
$$\bar{\psi}_{m}(x)=(-1)^m\int_{0}^{\infty}\left(\frac{1}{t}-\frac{1}{e^t-1}\right)t^me^{-xt}dt,$$ valid for $\Re(s)>-m.$
Equation (1.4) gives us the Mellin transform for our $g(t).$ Now proceeding with our choices for $f(t)$ and $g(t)$ in the beginning of the proof, and noting $\sin((\frac{1}{2}+it)\pi)=\cosh(\pi t),$ the functional equation $\xi(s)=\xi(1-s)$ with (1.4), the Mellin transform of $f(t),$ $\pi^{-s/2}x^{-s}\Gamma(s/2)/2,$ we have
$$(-1)^{m}e^{x}\int_{0}^{\infty}t^m \bar{\psi}_{m}(t)e^{-\pi t^2e^{4x}}dt=\frac{1}{2\pi i}\int_{c'-i\infty}^{c'+i\infty}\frac{\pi\xi(s)}{\sin(\pi s)}\frac{\Gamma(1-s+m)}{s(s-1)\Gamma(1-s)}x^{-s}ds,$$ for $0<c'<1.$
Put $c'=\frac{1}{2}$ and we find the right side is equal to
$$\frac{1}{2\sqrt{x}}\int_{-\infty}^{\infty}\frac{\Xi(t)}{(t^2+\frac{1}{4})\cosh{\pi t}}\frac{\Gamma(\frac{1}{2}-it+m)}{\Gamma(\frac{1}{2}-it)}x^{-it}dt.$$
Splitting this bilateral integral up into two integrals and replacing $x$ by $e^{2x}$ leads to the result. \end{proof}

Koshliakov [12, eq.(36), eq.(40)] gives a similar integral to (1.1) but squaring the integrand $\Xi(t)(t^2+\frac{1}{4})^{-1}.$ 
$$\int_{0}^{\infty}\frac{\Xi^2(t/2)}{(1+t^2)^2}\frac{\cos(xt)}{\cosh(\pi t/2)}dt=e^{x/2}\int_{0}^{\infty}K_{0}(2\pi e^{x}t)\Lambda(t)dt,$$
where 
$$\Lambda(t)=\zeta(2)+\gamma^2-2\gamma_1+2\gamma\log(t)+\frac{1}{2}\log^2(t)+\sum_{n\ge1}d(n)\left(\frac{1}{t+n}-\frac{1}{n}\right).$$
Here $d(n)$ denotes the number of divisors of $n,$ and the Stieltjes constant is 
$$\gamma_1=\lim_{N\rightarrow\infty}\left(\sum_{1\le n\le N}\frac{\log(n)}{n}-\frac{\log(N)^2}{2}\right).$$
Moll and Dixit state his result in [6, Theorem 4.4], and offer further interesting generalizations. However, we noticed a slightly different form of this integral based on the type of proof we used in the previous theorem. 

\begin{theorem} Let $K_s(t)$ denote the modified Bessel function. For $x>0,$
\begin{equation}\int_{0}^{\infty} (\psi(t+1)-\log(t))\left(\frac{2\pi e^{x/2}}{t}-e^{-x/2}4\sum_{n\ge1}K_{0}(2\pi nte^{-x})\right)dt=\int_{0}^{\infty}\frac{\Xi^2(t)}{(t^2+\frac{1}{4})^2}\frac{\cos(xt)}{\cosh(\pi t)}dt.\end{equation}
\end{theorem}

\begin{proof} First, since [9, pg.920, eq. WA 231, 245(9)] $K_{0}(t)=O(t^{-1/2}e^{-t}),$ for $t\rightarrow\infty$ in $|\arg(t)|<3\pi/2$ where as usual $K_s(t)$ is the modified Bessel function [9, pg.928], it follows that the 
series on the right side of the formula (by [9, pg. 676, eq. EH II 51(27)]) $d>0,$
$$\sum_{n\ge1}\frac{1}{2\pi i}\int_{d-i\infty}^{d+i\infty}\Gamma^2(\frac{s}{2})(\frac{x}{2n})^{-s}ds=4\sum_{n\ge1}K_{0}(n\pi x),$$
converges absolutely. Summing through, we find for $d'>1,$
$$\frac{1}{2\pi i}\int_{d'-i\infty}^{d'+i\infty}\Gamma^2(\frac{s}{2})(\frac{x}{2})^{-s}\zeta(s)ds=4\sum_{n\ge1}K_{0}(n\pi x),$$
which has a simple pole at $s=1.$ Note that $\lim_{s\rightarrow1}(s-1)\Gamma^2(\frac{s}{2})(\frac{x}{2})^{-s}\zeta(s)=\frac{2\pi}{x},$ since $\Gamma(\frac{1}{2})=\sqrt{\pi}$ and $\lim_{s\rightarrow1}(s-1)\zeta(s)=1.$ If we calculate this residue and move the line of integration to $\Re(s)=b,$ $0<b<1,$ we find
\begin{equation}\frac{1}{2\pi i}\int_{b-i\infty}^{b+i\infty}\Gamma^2(\frac{s}{2})(\frac{x}{2})^{-s}\zeta(s)ds+\frac{2\pi}{x}=4\sum_{n\ge1}K_{0}(n\pi x).\end{equation}
Now noting (1.3) and (1.6) in conjunction with Parseval's theorem for Mellin transforms, we obtain the result after applying the functional equation for $\xi(s).$ 
\end{proof}
Other associated integral formulae related to Fourier cosine transforms may be obtained from using known evaluations.
\begin{theorem} We have for $\frac{\pi}{2}>|\Im(\beta')|,$
\begin{equation}\int_{0}^{\infty}\frac{\cosh(t)}{\cosh(2t)+\cosh(2\beta')}\left(e^{t/2}-2e^{-t/2}\sum_{n\ge1}e^{-\pi n^2e^{-2t}}\right)dt\end{equation}
$$=\frac{\pi e^{\beta'}}{4\cosh(\beta')}\int_{0}^{\infty}(\psi(t+1)-\log(t))e^{-\pi t^2e^{4\beta'}}dt.$$
\end{theorem}
\begin{proof} We apply the classical [15, eq.(2.16.2)]
$$\int_{0}^{\infty}\frac{\Xi(t)}{(t^2+\frac{1}{4})}\cos(xt)dt=\frac{\pi}{2}(e^{x/2}-2e^{-x/2}\sum_{n\ge1}e^{-\pi n^2e^{-2x}}),$$
with an integral from [9, pg.511, eq. ET I 31(16)], valid for $\pi\Re(\alpha)>|\Im(\bar{\alpha}\beta)|,$ and
$$\int_{0}^{\infty}\frac{\cos(ty)\cosh(t\alpha/2)dt}{\cosh(\alpha t)+\cosh(\beta)}=\frac{\pi}{2\alpha}\frac{\cos(\beta y/\alpha)}{\cosh(\beta/2)\cosh(\pi y/\alpha)}.$$
We need only let $\beta\rightarrow2\beta',$ and $\alpha=2,$ and then we have $\frac{\pi}{2}>|\Im(\beta')|.$ This coupled with Parseval's theorem gives the result. \end{proof}

The integral formula in Theorem 1.3 may be useful for fast numerical approximations, since the integrand of the integral on the left side is comprised entirely of exponentials. A possible avenue would be to truncate the series at the first term, since it is a double exponential. Note that the integral on the left side of (1.7) hints to us about the non-negativity of the integral on the right side. We will investigate this further in the next section.

\section{Inequalities}  We first start with some preliminaries about some special functions which can be found in [1]. First we recall the error function [9, pg.887],
$$\erf(x)=\frac{2}{\sqrt{\pi}}\int_{0}^{x}e^{-t^2}dt,$$
 and its complement, $\erfc(x)=1-\erf(x).$ Also, $\erfi(x):=-i\erf(ix).$ We also use the exponential integral [9, pg.883]
 $$\Ei(x)=-\int_{-x}^{\infty}\frac{e^{-t}}{t}dt. $$
\begin{theorem} Define the Hardy integral for $y>0$ as
$I(y):=\int_{0}^{\infty}(\psi(t+1)-\log(t))e^{-yt^2}dt.$ Then we have,
$$\frac{(2c_1)^2\sqrt{y}}{\sqrt{\pi}\erfi(\sqrt{y})}-\frac{\Ei(-y)}{4}-\frac{e^{-y}}{12}+\frac{\sqrt{y\pi}}{12}\erfc(\sqrt{y})\le I(y).$$
Furthermore,
$$I(y)\le \sqrt{c_2\sqrt{\frac{\pi}{2y}}\erf(\sqrt{2y})}-\frac{\Ei(-y)}{4},$$
where $c_1=0.952894,$ and $c_2=1.56624.$
\end{theorem}

\begin{proof} We first need a result from the paper [13, Corollary 1], for $x>0,$
\begin{equation}\frac{1}{2x}-\frac{1}{12x^2}<\psi(x+1)-\log(x)<\frac{1}{2x}.\end{equation}
For the upper bound we first write,
$$I(y)=\int_{0}^{1}(\psi(t+1)-\log(t))e^{-yt^2}dt+\int_{1}^{\infty}(\psi(t+1)-\log(t))e^{-yt^2}dt.$$
Applying (2.1) we have that this is
$$\le \int_{0}^{1}(\psi(t+1)-\log(t))e^{-yt^2}dt-\frac{\Ei(-y)}{4}.$$
By Schwarz's inequality and the fact that 
$$\int_{0}^{1}(\psi(t+1)-\log(t))^2dt=1.56624$$ we have the upper bound. For the lower bound, we write the Schwarz's inequality in the form
$$\left(\int_{0}^{1}(\psi(t+1)-\log(t))^{1/2}dt\right)^2\le\left(\int_{0}^{1}(\psi(t+1)-\log(t))e^{-yt^2}dt\right)\left(\int_{0}^{1}e^{yt^2}dt\right),$$
and also note,
$$\int_{0}^{1}\sqrt{(\psi(t+1)-\log(t))}dt=0.952894.$$ This calculation, coupled with similar computations derived from the left side of (2.1) and the definitions of the special integral functions, gives our result.
\end{proof}

Now it can be seen that the lower bound actually implies the Hardy integral $I(x)$ is non-negative for $x>0,$ and also the squeeze theorem shows that $I(x)\rightarrow0$ when $x\rightarrow\infty.$ 

\section{Other observations and further comments}
While studies such as [6] show there is a plethora of integral formulae like Hardy's (1.1) to be found, it would be interesting to see if these integrals could be incorporated into similar work like Csordas [2]. We were interested in seeing if (1.1) could be adapted to imply information beyond Hardy's theorem, that $\Xi(t)$ has infinitely many real zeros. The classical ideas are included in Titchmarsh's text [15, pg. 256-260], and are based around the idea of first assuming $\Xi(t)$ is of one sign, and then arriving at a contradiction from properties of an integral of the form $\int_{0}^{\infty}\Xi(t)k(t)dt,$ $k(t)$ non-negative for $t>0.$ The idea in [15, pg. 260] considers the moments $k(t)=t^{2n}.$ Apply the operator $\bar{\partial}_x=-\partial_x^2+\frac{1}{4}$ of Glasser [8] to (1.1) (after replacing $t$ by $2t$) and set $x=x'+i\pi$ to get
$$\lim_{x\rightarrow x'+i\pi}\bar{\partial}_x(e^{x/2}I(\pi e^{2x}))=\int_{0}^{\infty}\Xi(t)\cos(x't)dt+i\int_{0}^{\infty}\Xi(t)\frac{\sin(x't)\sinh(\pi t)}{\cosh(\pi t)}dt.$$
Differentiating $2n$ times and setting $x'=0$ gives us the right side as the value $(-1)^n\int_{0}^{\infty}t^{2n}\Xi(t)dt.$ Note that the integral $I(x)$ clearly converges only for $\Re(x)>0,$ and hence $I(\pi e^{2z})$ converges when $-\frac{\pi}{4}<\Im(z)<\frac{\pi}{4}.$ It would be interesting if new information could be obtained about $\Xi(t)$ from alternative forms of these moments like those found here.

1390 Bumps River Rd. \\*
Centerville, MA
02632 \\*
USA \\*
E-mail: alexpatk@hotmail.com
\end{document}